\documentclass[12pt]{article}

\usepackage{amssymb}%
\usepackage{amsmath}%
\usepackage{amsthm}%

\usepackage{url}%

\usepackage{xcolor}%

\allowdisplaybreaks%

{\theoremstyle{plain}%
 \newtheorem{theorem}{Theorem}

}
{\theoremstyle{remark}

}
{\theoremstyle{definition}

\newtheorem{example}{Example}
}

\makeatletter
\newcommand{\dotDelta}{{\vphantom{\Delta}\mathpalette\d@tD@lta\relax}}
\newcommand{\d@tD@lta}[2]{%
 \ooalign{\hidewidth$\m@th#1\mkern-1mu\cdot$\hidewidth\cr$\m@th#1\Delta$\cr}%
}
\makeatother

\begin{document}

%\author

\begin{center}
{\large Applications of a class of transformations of complex sequences}
\end{center}

\begin{center}
{\textsc{John M. Campbell} } 

 \ 

\end{center}

\begin{abstract}
 Through an application of a remarkable result due to Mishev in 2018 concerning the inverses for a class of transformations of sequences of complex 
 numbers, we obtain a very simple proof for a famous series for $\frac{1}{\pi}$ due to Ramanujan. We then apply Mishev's transform to provide proofs for 
 a number of related hypergeometric identities, including a new and simplified proof for a family of series for $\frac{1}{\pi}$ previously obtained by Levrie 
 via Fourier--Legendre theory. We generalize this result using Mishev's transform, so as to extend a result due to Guillera on a Ramanujan-like 
 series involving cubed binomial coefficients and harmonic numbers. 
\end{abstract}


 \ 


\section{Introduction}
 Included in Ramanujan's first letter to Hardy was the now-famous formula 
\begin{equation}\label{mainRamanujan}
 \frac{2}{\pi} = \sum_{n=0}^{\infty} \left( - \frac{1}{64} \right)^{n} \binom{2n}{n}^{3} (4n+1) 
\end{equation}
 that had been previously proved by Bauer in 1859 \cite{Bauer1859} via a Fourier--Legendre expansion; see also 
 \cite{Campbell2021,ChuZhang2014,Levrie2010}, along with \cite{BaruahBerndtChan2009} and the references pertaining to \eqref{mainRamanujan} therein. 
 As suggested in \cite{GuilleraRogers2014}, Ramanujan's series evaluations as in \eqref{mainRamanujan} are of interest in part because of the close 
 association of such formulas with advanced subjects in number theory and complex analysis. In 1994, Zeilberger \cite{EkhadZeilberger1994} used the 
 \emph{Wilf--Zeilberger} (WZ) method \cite{PetkovsekWilfZeilberger1996} to produce a very simple and elegant computer proof of \eqref{mainRamanujan}. 
 In this article, we present a new, one-page proof of \eqref{mainRamanujan} that does not involve the WZ method. This new proof of ours relies on a 
 remarkable transformation identity for complex sequences introduced in 2018 by Mishev in what we consider to be an overlooked and underrated research 
 article \cite{Mishev2018} that does not seem to have been yet cited. Furthermore, we apply Mishev's transform to construct a simplified 
 proof for and a generalization of 
 an identity due to Levrie \cite{Levrie2010} on an infinite family of series for $\frac{1}{\pi}$ generalizing \eqref{mainRamanujan}, 
 and we apply our generalization of Levrie's identity to extend a remarkable result due to Guillera \cite{Guillera2013} 
 on a Ramanujan-like series involving harmonic numbers and cubed central binomial coefficients. 
 Letting $H_{m} = 1 + \frac{1}{2} + \cdots + \frac{1}{m}$ denote the $m^{\text{th}}$ harmonic number, 
 one of the main results in this article is our Ramanujan-like series formula 
\begin{equation}\label{20211214947PM1A}
 -\frac{2 \ln 2}{\pi} = \sum_{n = 0}^{\infty} \left( -\frac{1}{64} \right)^{n} \binom{2 n}{n}^3 (4 n + 1) H_{2 n}, 
\end{equation}
 noting the resemblance to Ramanujan's formula in \eqref{mainRamanujan}. 

\subsection{Mishev's transform} 
 We will use the Pochhammer symbol 
 $(\alpha)_{n} = \alpha (\alpha + 1) \cdots (\alpha + n - 1)$ for $n \in \mathbb{N}_{0}$
 and the usual notation for generalized hypergeometric series: 
\begin{equation*}
 {}_{p}F_{q} \left[ \begin{matrix} 
 a_{1}, a_{2}, \ldots, a_{p} \\ 
 b_{1}, b_{2}, \ldots, b_{q} \end{matrix} \ ; \ x \right] = 
 \sum_{n = 0}^{\infty} \frac{ \left( a_{1} \right)_{n} 
 \left( a_{2} \right)_{n} \cdots 
 \left( a_{p} \right)_{n} }{ \left( b_{1} \right)_{n} 
 \left( b_{2} \right)_{n} \cdots \left( b_{q} \right)_{n} } \frac{x^{n}}{n!}.
\end{equation*}
 We also write $$ \left[ \begin{matrix} \alpha_{1}, \alpha_{2}, \ldots, \alpha_{\ell_{1}} \vspace{1mm} \\ \beta_{1}, \beta_{2}, \ldots, \beta_{\ell_{2}} 
 \end{matrix} \right]_{n} = \frac{ \left( \alpha_{1} \right)_{n} \left( \alpha_{2} \right)_{n} \cdots \left( \alpha_{\ell_{1}} \right)_{n} }{ \left( \beta_{1} 
 \right)_{n} \left( \beta_{2} \right)_{n} \cdots \left( \beta_{\ell_{2}} \right)_{n}}. $$ The main transformation identity in \cite{Mishev2018} 
 may be formulated in an equivalent way as below. 

\begin{theorem}\label{theoremMishev}
 (Mishev, 2018) For a sequence $(x_{n} : n \in \mathbb{N}_{0})$ of complex numbers, 
 and for a parameter $a$ in $\mathbb{C} \setminus \{ 0, -1, -2, \ldots \}$, 
 the equality 
\begin{equation}\label{equationMishev}
 x_{n} = \frac{1}{n! (a + 1)_n}\sum _{k=0}^n 
 \left[ \begin{matrix} a, 1 + \frac{a}{2}, -n \vspace{1mm}\\ 
 1, \frac{a}{2}, 1 + a + n \end{matrix} \right]_{k} 
 \sum _{\ell=0}^k (-k)_\ell (k+a)_\ell x_{\ell} 
\end{equation}
 holds true ({\bf Theorem 3.1} in \cite{Mishev2018}). 
\end{theorem}

 This is proved in \cite{Mishev2018} by defining a transform 
 $ L_{a}(x)_{n} = 
 \sum_{k=0}^{n} (-n)_{k} (n + a)_{k} x_{k} $ 
 for a complex-valued sequence $(x_{n} )_{ n \in \mathbb{N}_{0}}$, 
 and by determining the inverse of this transform 
 using Dixon's classical summation theorem. 

\section{A new and simple proof of Ramanujan's formula}\label{sectionmain}
 Zeilberger's proof of \eqref{mainRamanujan} relied on a finite hypergeometric sum identity that was established via a computer-generated WZ proof 
 certificate, and Carlson's theorem \cite[p.\ 39]{Bailey1935} \cite{Weisstein} 
 was used to establish 
 that the corresponding identity obtained by replacing the upper limit
 of the aforementioned finite sum with $\infty$ holds for suitably bounded real $n$, 
 with \eqref{mainRamanujan} as a special case of this resultant infinite sum identity. It is unclear as to 
 how the WZ proof from \cite{EkhadZeilberger1994} may be altered so as to obtain our new results given in Section 
 \ref{sectionFurther}. 

 It is well-known \cite{BaruahBerndt2010,Chu2011Ramanujan,Chu2011Comp,ChuZhang2014,Hardy1924} 
  that Ramanujan's formula, as in \eqref{mainRamanujan}, may be  
  proved via   the limiting behaviour of a well-poised ${}_{5}F_{4}(1)$-series given by Dougall.  
 In  contrast, 
   our new proof of \eqref{mainRamanujan} 
     relies on a ${}_{4}F_{3}(-1)$-identity that we obtain using Mishev's work \cite{Mishev2018}.  

 Another way of proving \eqref{mainRamanujan} would be given by expressing the generating function for the sequence of cubed 
 central binomial coefficients via 
 Clausen's hypergeometric product (cf.\ \cite{BaileyBorweinBroadhurstGlasser2008}), so as to obtain an expression involving the square of the complete 
 elliptic integral of the first kind, and by then using an elliptic singular value, and some classical identities for the complete elliptic integrals. This sophisticated 
 approach is in contrast to our brief proof of Theorem \ref{maintheorem} below. Yet another proof of \eqref{mainRamanujan} was given in 
 \cite{LevrieNimbran2018} using recurrences and reindexing arguments for a family of ${}_{3}F_{2}(1)$-series, in contrast to our derivation of the 
 ${}_{4}F_{3}(-1)$-identity in \eqref{generalizedZeilberger}. 

 Yet another proof of \eqref{mainRamanujan} is given in \cite{Ojanguren2018} using a series acceleration technique due to Gosper applied to a classical 
 infinite product identity from Euler. Our proof of \eqref{mainRamanujan} below is dramatically different and considerably simpler.

\begin{theorem}\label{maintheorem}
 Ramanujan's formula for $\frac{1}{\pi}$ shown in \eqref{mainRamanujan} holds true. 
\end{theorem}

\begin{proof}
 We set $x_{\ell} = \frac{1}{(\ell!)^2}$ in Mishev's transformation identity, as in Theorem \ref{theoremMishev}. Since $k \geq 0$ is an integer, we see 
 that we may then rewrite the inner sum in \eqref{equationMishev} as 
\begin{equation}\label{Gaussapplies}
 \sum _{\ell=0}^\infty \frac{ (-k)_\ell (k+a)_\ell }{\left( \ell! \right)^{2} }, 
\end{equation}
 since $ (-k)_\ell $ vanishes for $k < \ell$. According to Gauss' famous ${}_{2}F_{1}(1)$-identity
\begin{equation*}
 {}_{2}F_{1} \left[ \begin{matrix} 
 a, b \\ 
 c \end{matrix} \ ; \ 1 \right] = 
 \frac{ \Gamma(c) \Gamma(c - a - b) }{ \Gamma(c - a) \Gamma(c-b)}, 
\end{equation*}
 we see that the infinite series in \eqref{Gaussapplies} 
 evaluates as $ \frac{(-1)^k (a)_k}{k!}$. So, this immediately gives us that Mishev's 
 identity, as in \eqref{equationMishev}, for the special case whereby $x_{\ell} = \frac{1}{(\ell!)^2}$, gives us that 
\begin{equation}\label{generalizedZeilberger}
 {}_{4}F_{3}\!\!\left[ 
 \begin{matrix}
 a, a, -n, 1 + \frac{a}{2} \vspace{1mm} \\ 
 1, \frac{a}{2}, n + a + 1
\end{matrix} \ ; \ -1 \right] = \frac{ (a+1)_{n} }{n!}, 
\end{equation}
 for nonnegative integers $n$ and for suitably bounded $a$ (cf.\ \cite{EkhadZeilberger1994}). 
 As in \cite{EkhadZeilberger1994}, by Carlson's 
 theorem, we find that \eqref{generalizedZeilberger} also holds for suitably bounded non-integer values for $n$. In particular, setting $a = 
 \frac{1}{2} $ and $ n = - \frac{1}{2}$, this immediately gives us that 
\begin{equation}\label{preciselysame}
 {}_{4}F_{3}\!\!\left[ 
 \begin{matrix}
 \frac{1}{2}, \frac{1}{2}, \frac{1}{2}, \frac{5}{4} \vspace{1mm} \\ 
 \frac{1}{4}, 1, 1 
 \end{matrix} \ ; \ -1 \right] = \frac{2}{\pi}, 
\end{equation}
 and this is precisely the same as Ramanujan's formula in \eqref{mainRamanujan}, 
 rewriting the summand of the left-hand side of \eqref{preciselysame} using central binomial coefficients. 
\end{proof}

 The above application of Mishev's transform does not appear in \cite{Mishev2018}, and does not seem to be present in any relevant literature. As stated 
 above, Mishev's article \cite{Mishev2018} does not seem to have been cited previously, as of this note being written. This, together with our above 
 application of Theorem \ref{theoremMishev}, suggests that Mishev's 2018 article in \cite{Mishev2018} is underappreciated. It also appears that no 
 equivalent or meaningfully similar formulation of our proof of \eqref{mainRamanujan} has appeared in any in any past literature. 
 For research relevant to 
 Zeilberger's WZ proof in \cite{EkhadZeilberger1994}, we refer to 
 Guillera's work in WZ theory \cite{Guillera2006,Guillera2016,Guillera2010} 
 and past work on supercongruences concerning Ramanujan-type sums \cite{Guo2018,Wang2018,Zudilin2009}. 

\section{Further results}\label{sectionFurther}
 The hypergeometric identity in \eqref{generalizedZeilberger} is of interest in its own right, 
 and motivates the use of Mishev's transform \cite{Mishev2018} with inputted sequences given by variants 
 of the $x$-function involved in our proof of Theorem \ref{maintheorem}. 

\begin{example}
 Setting $x_{\ell} = \frac{1}{\ell! (\ell + 1)!}$ in the formulation of Mishev's transform given in 
 Theorem \ref{theoremMishev}, we may obtain that 
\begin{equation*}
 {}_{4}F_{3}\!\!\left[ 
 \begin{matrix}
 a, a - 1, -n, 1 + \frac{a}{2} \vspace{1mm} \\ 
 2, \frac{a}{2}, n + a + 1
\end{matrix} \ ; \ -1 \right] = \frac{ \left( a + 1 \right)_{n} }{(n + 1)!} 
\end{equation*}
 for suitably bounded $a$ and $n$; we may again apply Carlson's theorem, as above. 
\end{example}

\begin{example}
 Setting $x_{\ell} = \frac{1}{\ell! \left(\ell + \frac{1}{2}\right)!}$, 
 Mishev's transform gives us that 
\begin{equation*}
 {}_{4}F_{3}\!\!\left[ 
 \begin{matrix}
 a, -n, \frac{a}{2} + 1, a - \frac{1}{2} \vspace{1mm} \\ 
 \frac{3}{2}, \frac{a}{2}, a + n + 1
\end{matrix} \ ; \ -1 \right] = \frac{ \left( a + 1 \right)_{n} \, \sqrt{\pi} }{ 2 \left( n + \frac{1}{2} \right)! } 
\end{equation*}
 for suitably bounded $a$ and $n$. 
 For example, we may obtain the exotic hypergeometric series evaluation by setting $a = \frac{1}{4}$ and $n = \frac{1}{4}$: 
\begin{equation*}
 {}_{4}F_{3}\!\!\left[ 
 \begin{matrix}
 -\frac{1}{4}, -\frac{1}{4}, \frac{1}{4}, \frac{9}{8} \vspace{1mm} \\ 
 \frac{1}{8}, \frac{3}{2}, \frac{3}{2} 
\end{matrix} \ ; \ -1 \right] = {\frac {2\,\sqrt {2}}{3}}. 
\end{equation*}
\end{example}

 We find it appropriate to highlight the following evaluation for a ${}_{4}F_3(-1)$-series with 
 three real parameters as a theorem. 

\begin{theorem}\label{threereal}
 The evaluation 
\begin{equation*}
 {}_{4}F_{3}\!\!\left[ 
 \begin{matrix}
 a, -n, \frac{a}{2} + 1, a - b + 1 \vspace{1mm} \\ 
 b, \frac{a}{2}, 1 + a + n 
\end{matrix} \ ; \ -1 \right] 
 = \left[ \begin{matrix} a + 1 \vspace{1mm} \\ b 
 \end{matrix} \right]_{n} 
\end{equation*}
 holds true for suitably bounded real $a$, $b$, and $n$. 
\end{theorem}

\begin{proof}
 This follows in a direct way from Mishev's transform, as given in Theorem \ref{theoremMishev}, 
 by setting $x_{\ell} = \frac{1}{ \left( 1 \right)_{\ell} \left( b \right)_{\ell} }$
 for a real parameter $b$, and by then applying Carlson's theorem. 
\end{proof}

 By again setting $a = \frac{1}{2}$ and $n = -\frac{1}{2}$, recalling our proof of Ramanujan's formula in \eqref{mainRamanujan}, Theorem \ref{threereal} 
 then gives us an infinite family of generalizations of this famous formula. 

\subsection{An identity due to Levrie}
 We offer a simplified proof and a generalization of the following identity that had been introduced by Levrie and proved in 
 \cite{Levrie2010} via Fourier--Legendre theory. The below proof is remarkably simple and different compared to Levrie's proof, noting that we do not 
 require orthogonal polynomials, integration identities for Legendre polynomials, etc. 

\begin{theorem}\label{theoremLevrie2010}
 (Levrie, 2010) The identity 
\begin{align*}
 & \sum_{i = 0}^{\infty}
 (-1)^{i} \frac{ \left( \frac{1}{2} \right)_{i}^{3} }{ (i!)^{3} }
 \frac{ 4 i + 1}{ (2 i - (2 k - 1)) \cdots (2 i - 1)
 (i + 1) \cdots (i + k) } \\ 
 & = (-1)^{k} \frac{1}{ 2^{k-1} \left( \frac{1}{2} \right)_{k}^{2} } \frac{1}{\pi}
\end{align*}
 holds for all natural numbers $k$ \cite{Levrie2010}. 
\end{theorem}

\begin{proof}
 Again, we set $a = \frac{1}{2}$ and $n = -\frac{1}{2}$ in Theorem \ref{threereal}, giving us that 
\begin{equation}\label{Levriegeneralization}
 {}_{4}F_{3} \left[ 
 \begin{matrix}
 \frac{1}{2}, \frac{1}{2}, \frac{5}{4}, \frac{3}{2} - b \vspace{1mm} \\ 
 \frac{1}{4}, 1, b 
 \end{matrix} \ ; \ -1 \right] 
 = \frac{2 \, \left( b - \frac{1}{2} \right)_{\frac{1}{2}} }{\sqrt{\pi}} 
\end{equation}
 for suitably bounded real $b$. 
 For the special case whereby $b$ is a natural number, and letting $i \in \mathbb{N}_{0}$, we find that 
\begin{equation}\label{bPochhammerquotient}
 \frac{ \left( \frac{3}{2} - b \right)_{i} }{ \left( b \right)_{i} } 
\end{equation}
 may be written as 
\begin{align*}
 & (-1)^{b+1} \frac{ \left( \frac{ (2(b-1))! }{2^{b-1}} \right) }{ (2i-1) \cdots (2 i - 2 (b-1) + 1 ) 
 (i + 1) \cdots (i + b - 1 ) } \frac{\binom{2i}{i}}{4^i}. 
\end{align*}
 Rewriting the summand on the left-hand side of \eqref{Levriegeneralization} accordingly, we obtain an equivalent form
 of the desired identity due to Levrie. 
\end{proof}

 We may regard \eqref{Levriegeneralization} as generalizing Levrie's identity, as given in Theorem \ref{theoremLevrie2010}, for real values $b$. 

\subsection{Extensions of a series evaluation due to Guillera}
 A remarkable Ramanujan-like series due to Guillera \cite{Guillera2013} that involves harmonic numbers is reproduced in an equivalent form as below, noting 
 the close resemblance to Ramanujan's series \eqref{mainRamanujan}: 
\begin{equation}\label{expandedGuillera}
 \frac{1}{6 \sqrt{2} \pi } \left(\frac{\Gamma \left(\frac{1}{8}\right) 
 \Gamma \left(\frac{3}{8}\right)}{\Gamma \left(\frac{1}{4}\right) 
 \Gamma \left(\frac{3}{4}\right)}\right)^2 - \frac{4 \ln (2)}{\pi }
 = \sum _{n=0}^{\infty} \left(-\frac{1}{64}\right)^n \binom{2 n}{n}^3 (4 n+1) H_n. 
\end{equation}
 This is proved in \cite{Guillera2013} using the identity 
$$ \frac{1}{2} \sum_{n=0}^{\infty} (-1)^{n} \frac{ \left( \frac{1}{2} + x \right)^{3}_{n} }{\left( 1 \right)_{n}^{3} } (1 + 2 x + 4 n) 
 = \frac{\cos \pi x}{\pi}$$
 obtained in \cite{Chu2011Comp} together with an expansion 
 for $$ \frac{1}{2} \sum_{n=0}^{\infty} (-1)^{n} \frac{ \left( \frac{1}{2} + x \right)^{3}_{n} }{ \left(1 + x \right)_{n}^{3} } 
 \left( 1 + 4 x + 4 n \right) $$
 offered in \cite{Guillera2007}. It is unclear as to how this kind of approach may be generalized so as to evaluate series involving 
\begin{equation}\label{cubedbinomialDelta}
 \left( \frac{1}{64} \right)^{n} \binom{2n}{n}^3 \Delta(n) 
\end{equation}
 more generally, letting $(\Delta(n) : n \in \mathbb{N}_{0})$ denote a sequence of harmonic-type numbers. 
 This has motived our recent work in \cite{Campbell2021} on the use of the WZ method 
 in the evaluation of a Ramanujan-inspired series involving \eqref{cubedbinomialDelta}. 
 However, it is unclear as to how the methodologies from \cite{Campbell2021,Guillera2013} 
 may be applied to evaluate the series obtained by replacing the harmonic factor $H_{n}$ in 
 \eqref{expandedGuillera} with, say, $H_{2n}$ or $O_{n} = 1 + \frac{1}{3} + \cdots + \frac{1}{2n-1}$. 
 We have successfully solved this problem, through an application of our generalization, as in \eqref{Levriegeneralization}, 
 of Levrie's identity. It appears that known integration and coefficient-extraction 
 techniques from References as in \cite{Campbell2019,ChuCampbell2021,WangChu2020} for evaluating series involving 
\begin{equation*}
 \left( \frac{1}{16} \right)^{n} \binom{2n}{n}^2 \Delta(n) 
\end{equation*}
 cannot be applied to the much more difficult problem of evaluating harmonic sums involving \emph{cubed}, as opposed to squared, 
 central binomial coefficients. This emphasizes the remarkable nature about our results as in 
 Theorem \ref{newRamanujan} below. 

 Note that the digamma function is such that the following holds 
 (cf.~\cite{Rainville1960}, \S9): 
\begin{equation}\label{psidef}
 \psi(z) = \frac{d}{dz}\ln\Gamma(z)=\frac{\Gamma'(z)}{\Gamma(z)} =-\gamma+\sum_{n=0}^{\infty}\frac{z-1}{(n+1)(n+z)}, 
\end{equation}
 with $\gamma$ the Euler--Mascheroni constant. 

\begin{theorem}\label{newRamanujan}
 The Ramanujan-like series evaluation 
 $$ -\frac{2 \ln 2}{\pi} = \sum_{i = 0}^{\infty} \left( -\frac{1}{64} \right)^{i} \binom{2i}{i}^3 (4 i +1) H_{2i} $$
 holds true. 
\end{theorem}

\begin{proof}
 Applying $\frac{d}{db}$ to the summand factor highlighted in \eqref{bPochhammerquotient}, we obtain: 
\begin{align*}
 & -{\frac {1}{ \left( b \right)_{i} } \left( {\frac{3}{2}}-b \right)_{i} \left( \psi \left( i+{\frac{3}{2}}- 
 b \right) -\psi \left( {\frac{3}{2}}-b \right) \right) } - \\ 
 & {\frac {\psi 
 \left( i+b \right) -\psi \left( b \right) }{ \left( b \right)_{i} } \left( {\frac{3}{2}} - b \right)_{i} }. 
\end{align*}
 So, applying term-by-term differentiation to both sides of our generalization of Levrie's identity, with respect to the free variable in 
 \eqref{Levriegeneralization}, and then setting $b = 1$, this gives us that the sum of all expressions of the form $$ {\frac { \left( 4\,i+1 \right) \left( -1 
 \right)^{i + 1}}{{\pi}^{{\frac{3}{2}}} \left( i! \right) ^{3}} \left( \psi \left( i+{\frac{1}{2}} \right) +2\,\gamma+2\,\ln \left( 2 \right) +\psi \left( i+1 \right) 
 \right) \left( \left( i-{\frac{1}{2}} \right) ! \right) ^{3}} $$ for $i \in \mathbb{N}_{0}$ admits the following closed-form evaluation: $4\,{\frac {\ln \left( 2 
 \right) }{\pi}}$. We may find that $\psi\left( i + \frac{1}{2} \right) = 2 O_{i} - \gamma - 2 \ln 2$, in view of \eqref{psidef}. This, together with the equality 
 $\psi(i + 1) = H_{i} - \gamma$, give us that $$ \sum_{i=0}^{\infty} (-1)^{i+1} (2 O_{i} + H_{i} ) (4 i + 1) \frac{ \left( \left( i - \frac{1}{2} \right)! 
 \right)^{3} }{\pi^{3/2} (i!)^{3}} = 4\,{\frac {\ln \left( 2 \right) }{\pi}}. $$ Making use of the Legendre duplication formula
 \cite[p.\ 256]{AbramowitzStegun1992}, this gives us that $$ 
 \sum_{i=0}^{\infty} \left( - \frac{1}{64} \right)^{i} \binom{2i}{i}^3 (4 i + 1) (2 O_{i} + H_{i}) = -4\,{\frac {\ln \left( 2 \right) }{\pi}}. $$ In view of the 
 relation $ O_{i} = H_{2 i} - \frac{1}{2} H_i$, we obtain the desired result. 
\end{proof}

 We see that the above Theorem, in conjunction with Guillera's formula in \eqref{expandedGuillera}, also allows us to evaluate both 
$$ \sum_{i = 0}^{\infty} \left( -\frac{1}{64} \right)^{i} \binom{2i}{i}^3 (4 i +1) H_{2i}' $$
 and 
$$ \sum_{i = 0}^{\infty} \left( -\frac{1}{64} \right)^{i} \binom{2i}{i}^3 (4 i +1) O_{i}, $$
 letting $H_{n}' = 1 - \frac{1}{2} + \cdots + \frac{(-1)^{n+1}}{n}$ denote the $n^{\text{th}}$ alternating harmonic number. 
 It appears that Abel-type summation lemmas, following the techniques from \cite{Campbell2021}, 
 cannot be used to evaluate the above summations. 

 We are actively involved in exploring the application of Mishev's transform in conjunction with 
  WZ theory, and through the use of the modified Abel lemma on summation by parts. 
 Also, by setting $x_{\ell} = \frac{y^{\ell}}{ (1)_{\ell}^{2} }$ in 
 Theorem \ref{theoremMishev} for a variable $y$, 
 this gives us that the inner sum in \eqref{equationMishev} 
 may be written as a Jacobi polynomial; 
 this motivates the application of Mishev's transform 
 using the theory of orthogonal polynomials.

 \ 

John M.\ Campbell

York University

Toronto, Ontario, Canada

 \ 

{\tt jmaxwellcampbell@gmail.com}

\end{document}